\documentclass{amsart}
\usepackage{amsmath,amsthm,amssymb,amsthm, amsfonts}
\makeatletter
 \def\@makefnmark{%
 \leavevmode
 \raise.9ex\hbox{\check@mathfonts
 \fontsize\sf@size\z@\normalfont%
 \@thefnmark}%
 }
 \makeatother
\newcommand\diam{\operatorname{diam}}
\theoremstyle{definition}

\newtheorem{clm}{Claim}
\newtheorem{lem}{Lemma}[section]
\newtheorem{prop}[lem]{Proposition}
\newtheorem{thm}[lem]{Theorem}
\newtheorem{cor}[lem]{Corollary}
\newtheorem{ex}[lem]{Example}

\newcounter{cn}
\title{An indecomposable continuum as subpower Higson corona}
\author{Yutaka Iwamoto}
\begin{document}
\begin{abstract}
In this paper, we study topological properties of the subpower Higson coronas
 of proper metric spaces 
 and show that the subpower Higson corona of the half open interval
 with the usual metric
 is an indecomposable continuum.
Continuous surjections from Higson-type coronas onto a Higson-type compactifications of the half open interval are also constructed.
\end{abstract}
\maketitle
\renewcommand{\thefootnote}{\fnsymbol{footnote}}
\section{Introduction}
Let $(X, d)$ be a metric space with a metric $d$ and let $B_{d}(x,r)$
 denote the closed ball of radius $r$ centered at $x\in X$.
A metric $d$ on $X$ is called {\it proper} if all balls $B_{d}(x, r)$ are compact.
The Higson compactification is a compactification
 defined by the coarse structure of a proper metric space
 and plays an important role in the large-scale geometry \cite{Roe}.
Also the sublinear Higson compactification is known as a compactification
 defined by the sublinear coarse structure of a proper metric space
 \cite{CDSV}, \cite{DS}.
There are several ways to define a Higson type compactification,
 by a coarse structure  \cite{DKU}, \cite{DS}, \cite{Roe},
 by a large scale structure \cite{DH},
 and
 by a closed subring of the algebra
 of all continuous bounded real-valued functions \cite{Keesling}, \cite{KZ1}.
The subpower Higson compactification was introduced in \cite{KZ1}
 as a compactification defined by a closed subring of the algebra of all continuous bounded real-valued functions.
And the asymptotic power dimension was studied in \cite{KZ2}
 related to the subpower Higson corona.

In this paper, we adopt the definition associated with a closed subrings
 of the algebra of all continuous bounded real-valued functions
 and study topological properties of the subpower Higson compactifications and their coronas.
It is known that a Higson-type corona $\nu X$ can be realized by a discrete subspace of
 a proper metric space $X$.
The Higson corona and the subpower Higson corona contain a copy of
 the Stone-\v{C}ech remainder $\beta \mathbb N \setminus \mathbb N$
 of the natural numbers $\mathbb N$
 \cite{Keesling}, \cite{KZ1}.
Also it was shown that
 the Higson corona of the half open interval $[0, \infty)$ 
 with the usual metric is,
 like the Stone-\v{C}ech remainder $\beta [0,\infty) \setminus [0,\infty)$ \cite{Bellamy},
 an indecomposable continuum \cite{IT}.
We show that the subpower Higson corona of the half open interval with the usual metric
 is also an indecomposable continuum.
Then we construct continuous surjections from Higson-type coronas
 onto Higson-type compactifications of the half open interval.
%%%%%%%%%%%%%%%%%%%%%%%%%%%%%%%%%%%%%%%%%%%%%%%%%%%%%%%%%%%%%%%%%%%%%%%%%%%%%%%%%%

\section{Basic properties of the subpower Higson compactification}

A (not necessarily continuous) function $s: \mathbb R_+ \to \mathbb R_+$
 between the set of positive real numbers
 is called {\it asymptotically subpower}
 (resp. {\it asymptotically sublinear})
 if for every $\alpha>0$
 there exists $t_0 >0$ such that $s(t)<t^{\alpha}$
 (resp. $s(t)<\alpha t$)
 whenever $t>t_0$.
The set of all asymptotically subpower functions
 (resp. all asymptotically sublinear functions, all positive constant functions)
  is denoted by $\mathcal{P}$ (resp. $\mathcal{L}$, $\mathcal{H}$).
Since every constant functions are asymptotically subpower
 and every asymptotically subpower functions are asymptotically sublinear,
 it follows that $\mathcal{H}\subset \mathcal{P}\subset \mathcal{L}$.

In what follows,
 a metric space $(X,d)$ is assumed to have a base point $x_0$,
 and the distance $d(x_0, x)$ of $x\in X$ from $x_0$ is denoted by $|x|$.
For a subset $A$ of $X$, the diameter of $A$ is denoted by $\diam_{d} A$,
 that is, $\diam_{d} A=\sup \{ d(x,y)\mid x,y\in A\}$.

Let $(X,d)$ and $(Y,\rho )$ be proper metric spaces.
A map $f: X\to Y$ is called {\it Higson subpower}
 (resp. {\it Higson sublinear}, {\it Higson})
 provided that
\begin{align}
\tag*{$(\ast)_{s}^f$}
 \qquad \lim_{|x| \to \infty} \diam_{\rho} f(B_{d}(x, s(|x|)))=0
\end{align}
 for each $s\in \mathcal{P}$ (resp. $s\in \mathcal{L}$, $s\in \mathcal{H}$).
Hence, a map $f:X\to Y$ is Higson subpower
 if and only if, given asymptotically subpower function $s$
  and $\varepsilon>0$, there exists a compact subset $K\subset X$
  such that $\diam_{\rho} f(B_{d}(x, s(|x|))<\varepsilon$ whenever $x\in X\setminus K$. 

Let $C^{\ast}(X)$ be the set of all bounded real-valued continuous functions on $X$.
For a subset $F$ of $C^{\ast}(X)$,
 let $e_F :X\to \prod_{f\in F} I_f$,
  where $I_f = [\inf f, \sup f]$,
  $f\in F$,
 be the {\it evaluation map} of $F$ defined by $(e_F (x) )_f =f(x)$ for every $x\in X$.
It is known that
 if $F\subset C^{\ast}(X)$ separates points from closed sets,
 then $e_F$ is an embedding \cite[8.16]{Wi}.
Identifying $X$ with $e_F (X)$,
 the closure $\overline{e_F (X)}$ of $e_F (X)$ in
  $\prod_{f\in F} I_f$
 gives a compactification of $X$.

For a proper metric space $X$,
 we consider the following three subsets of $C^{\ast}(X)$:
\begin{align*}
C_{H}(X)&=\{ f\in C^{\ast}(X)\mid f \text{ satisfies } (\ast)_{s}^f \text{ for every } s\in \mathcal{H} \},\\
C_{P}(X)&=\{ f\in C^{\ast}(X)\mid
 f \text{ satisfies } (\ast)_{s}^f \text{ for every } s\in \mathcal{P} \},\\
C_{L}(X)&=\{ f\in C^{\ast}(X)\mid
 f \text{ satisfies } (\ast)_{s}^f \text{ for every } s\in \mathcal{L} \}.
\end{align*}
Then they are closed subrings of $C^{\ast}(X)$ with respect to the sup-metric.
Also, they contain all constant maps and separate points from closed sets.
Hence the subrings
 $C_{H}(X)$, $C_{P}(X)$ and $C_{L}(X)$
 determine compactifications of $X$.
Since $\mathcal{H}\subset \mathcal{P}\subset \mathcal{L}$,
 it follows that
 $C_L (X)\subset C_P(X)\subset C_H(X)$ 
 (cf. \cite{KZ1}).

Let $c_1 X$ and $c_2 X$ be compactifications of $X$.
We say $c_1 X\succeq c_2 X$ provided that
 there is a continuous map $f:c_1 X \to c_2 X$ such that $f|_{X}=\mbox{id}_X$. 
We note that a continuous map $f:c_1 X \to c_2 X$ with $f|_{X}=\mbox{id}_X$
 is unique and surjective.
If $c_1 X \preceq  c_2 X$ and $c_1 X \succeq c_2 X$,
 then we say that $c_1 X$ and $c_2 X$ are {\it equivalent}.

The {\it Higson compactification} $h_H (X)$,
 the {\it subpower Higson compactification} $h_P(X)$
 and 
 the {\it sublinear Higson compactification} $h_L(X)$
 are compactifications defined by closed subrings
 $C_{H}(X)$, $C_{P}(X)$ and $C_{L}(X)$ respectively,
 that is,
 $h_A (X)$ is 
 equivalent to $\overline{e_{C_A}(X)}$,
 where $A\in \{H, P, L\}$ and $C_A =C_A (X)$.
Their {\it coronas} are defined by
 $\nu_A X=h_A (X)\setminus X$ for each $A\in \{ H, P, L\}$.
These three compactifications are referred to as
 the {\it Higson type compactifications}.
We note that $h_L (X)\preceq h_P (X) \preceq h_H(X)$
 since $C_L (X)\subset C_P(X)\subset C_H(X)$.

One of the basic property of
 the Higson type compactifications is the following:
\begin{itemize}
\item[$(\natural)$]
A bounded continuous map $f: X\to \mathbb R$ has an extension
 $\hat{f}: h_A (X)\to \mathbb R$ if and only if $f \in C_{A}(X)$,
\end{itemize}
 where $A\in \{ H, P, L\}$.
This condition holds for any closed subring of $C^{\ast}(X)$ which
 contains all constant maps and separates points from closed sets \cite[Problem 3.12.22(e)]{Eng}.
See \cite{BY} for a comprehensive and detailed description of this property.

In this section,
 we shall derive some basic properties concerning the Higson subpower compactifications
 using the basic ideas of \cite{DKU}, \cite{DS} and \cite{Keesling}.

The following proposition is a fundamental property of the Higson type compactification
 which is derived from the property $(\natural)$
 and can be shown by a similar argument
 to that in \cite[Proposition 1]{Keesling}.

\begin{prop}[cf. \cite{Keesling}]
\label{P2-1}
Let $X$ be a proper metric space.
Then the subpower (resp. sublinear)
 Higson compactification is the unique compactification of $X$
 such that if $Y$ is any compact metric space and $f:X\to Y$ is continuous,
 then $f$ has a continuous extension $\hat{f} : h_{P}(X)\to Y$ 
 (resp. $\hat{f} : h_{L}(X)\to Y$)
 if and only if
 $f$ satisfies $(\ast )^{f}_{s}$ for any $s\in \mathcal{P}$
 (resp. for any $s\in \mathcal{L}$).
\end{prop}

Let $(X,d)$ be a metric space.
We call a finite system $E_1, \dots , E_n$ of closed subsets of $X$ 
 {\it diverges as a power function}
if there exist $\alpha>0$ and $r_0 >0$ such that
 $\max\{ d(x, E_i) \mid 1\leq i \leq n\}\geq |x|^\alpha$ whenever $|x|>r_0$.
If a finite system $E_1, \dots , E_n$ diverge as a power function
 then $(\cap_{i=1}^n E_{i})\cap (X\setminus B_d(x_0, r))=\emptyset$
 for some $r>0$.
As in the case of Higson compactification \cite[Lemma 2.2]{DKU}, we have the following:

\begin{lem}[cf. \cite{DKU}]
\label{L2-2}
Let $(X,d)$ be a non-compact proper metric space.
Let $E_1, \dots, E_n$ be a finite system of closed subsets of $X$
 such that $\cap_{i=1}^n E_i =\emptyset$.
Let $f_i : X\to \mathbb R_+ $ be the map defined by $f_i(x)=d(x, E_i)$, $1\leq i\leq n$
 and let $F=\Sigma_{i=1}^{n} f_i$.
If the system $E_1, \dots, E_n$ diverges as a power function
 then the map $g_i =f_i/F : X\to [0,1]$ is Higson subpower.
\end{lem}

\begin{proof}
First we note that the well-definedness of $g_i$
 follows from the assumption $\cap_{i=1}^n E_i =\emptyset$.
Let $\alpha, r_1 >0$ be positive numbers such that
 $\max\{ d(x, E_i) \mid 1\leq i \leq n\}\geq |x|^{2 \alpha}$ whenever $|x|>r_1$.
Then $F(x)\geq |x|^{2 \alpha}$ whenever $|x|>r_1$.
Let $s: \mathbb R_+ \to \mathbb R_+$ be an asymptotically subpower function.
Then there exists $r_2>0$ such that $s(t)< t^{\alpha}$ for every $t>r_2$.
Given $\varepsilon >0$,
 we can take $r_3 >0$ so that $(n+1)/t^{\alpha}<\varepsilon/2$ for every $t>r_3$.

Put $r_0 =\max \{ r_1, r_2, r_3\}$.
Let $x$ be a point with $|x|>r_0$ and let $y \in B_d(x, s(|x|))$.
Then we have
\begin{align*}
|g_i (x)-g_i(y)|&=\left| \dfrac{f_i (x)}{F(x)} -\dfrac{f_i (y)}{F(y)}\right|
 =\left| \dfrac{f_i (x)-f_i (y)}{F(x)}
 + \dfrac{f_i(y)}{F(y)}\cdot \dfrac{F (y)-F(x)}{F(x)}\right|\\
 &\leq \left| \dfrac{d(x,y)}{F(x)} \right|
+ \left| \dfrac{F (y)-F(x)}{F(x)}\right|
 \leq \dfrac{(n+1)s(|x|)}{F(x)}\\
 &\leq \dfrac{(n+1)s(|x|)}{|x|^{2\alpha}}
 <\dfrac{(n+1) |x|^{\alpha}}{|x|^{2\alpha}}=\dfrac{(n+1)}{|x|^{\alpha}}<\dfrac{\varepsilon}{2}.
\end{align*}
Hence, $\diam g_i (B_d(x, s(|x|)))\leq \varepsilon$ whenever $|x|>r_0$.
Thus $g_i$ is Higson subpower for every $1\leq i\leq n$.
\end{proof}

For a subset $A$ of a proper metric space $X$,
 $\overline{A}$ denotes the closure of $A$
 in the subpower Higson compactification $h_P X$ of $X$.
As in case of the Higson corona \cite[Proposition 2.3]{DKU} and of
 the sublinear Higson corona \cite[Lemma 2.3]{DS},
 we have the following.

\begin{prop}[cf. \cite{DKU}, \cite{DS}]
\label{P2-3}
Let $(X,d)$ be a non-compact proper metric space.
For a system $E_1,\dots , E_n$ of closed subsets of $X$,
 the following conditions are equivalent:
\begin{enumerate}
\item $\nu_P X \cap \left( \cap_{k=1}^{n}\overline{E_k} \right) =\emptyset$,
\item the system $E_1, \dots, E_n$ diverges as a power function.
\end{enumerate}
\end{prop}

\begin{proof}
Suppose that the condition (2) does not hold.
Then there exists a sequence $(x_i)_{i=1}^{\infty}$ in $X$ such that
$\lim_{i\to \infty} |x_i|=\infty$,
$|x_i|<|x_{i+1}|$
 and $d(x_i, E_k )<|x_i|^{1/i}$
 for every $i\in \mathbb N$ and $1\leq k \leq n$.
Let $p$ be a cluster point of the sequence $(x_i)_{i=1}^{\infty}$.
Then $p\in \nu_P X$.
Put $t_i =|x_i|$ for each $i\in \mathbb N$.
Let  $s:\mathbb R_+ \to \mathbb R_+$ be the (non-continuous) function
  defined by $s(t)= |x_i|^{1/i}$ when $t\in [t_i, t_{i+1})$ and $s(t)=1$ when $t\in [0, t_1)$.
It is easy to see that $s$ is an asymptotically subpower function
 and $d(x_i, E_k )<s(|x_i|)$ for every $i\in \mathbb N$ and $1\leq k \leq n$.
Hence, if $f$ is any element of $C_P(X)$,
 the condition $(\ast)_{s}^f$ implies that
 the distance between $f(x_i) $ and $f(E_k)$
 tends to zero as $i$ tends to infinity.
Considering the evaluation map $e_{C_P(X)}: X\to \prod_{f\in C_P(X)} I_f$,
 it follows from this fact that $p\in \overline{E_k}$
 for every $1\leq k\leq n$. 
Thus $p\in \nu_P X \cap \left( \cap_{k=1}^{n}\overline{E_k} \right) \neq \emptyset$.

Suppose that
 the system $E_1, \dots, E_n$ diverges as a power function.
Then we have
 $(\cap_{k=1}^n E_{k})\cap (X\setminus B_d(x_0, r))=\emptyset$
 for some $r>0$.
Put $F_k =\mbox{cl}_X (E_k \setminus B_d(x_0, r+1))$ for each $k =1,\dots, n$.
Then the system $F_1,\dots , F_n$
 diverges as a power function and $\cap_{k=1}^{n}F_k=\emptyset$.
Note that $\overline{F_k}\cap \nu_P X =\overline{E_k}\cap \nu_P X$ for every $1\leq k \leq n$.
By Lemma \ref{L2-2} there exist  Higson subpower maps $g_k : X\to \mathbb [0,1]$, $1\leq k\leq n$,
 such that $F_k \subset g_k^{-1}(0)$ and $\Sigma_{k=1}^{n} g_k =1$.
By Proposition \ref{P2-1},
 there exists an extension $G_k : h_P X \to [0,1]$ of $g_k$.
Then $\Sigma_{k=1}^{n} G_k =1$ and $\overline{F_k}\subset G_{k}^{-1}(0)$.
However,
 the condition $\Sigma_{k=1}^{n} G_k =1$ leads that
 $\nu_P X \cap \left( \cap_{k=1}^{n}\overline{E_k} \right) 
\subset \cap_{k=1}^n \overline{F_k}\subset \cap_{k=1}^n G_{k}^{-1}(0)=\emptyset$.
\end{proof}

We note that the implication $(1) \Rightarrow (2)$
 for the case $n=2$
 of the above proposition was in fact shown by Kucab and Zarichnyi in \cite{KZ3}.

\begin{prop}[cf. \cite{Keesling}]
\label{P2-4}
Let $X$ be a proper metric space and
 let $Y$ be a closed subset of $X$ with the induced metric from $X$.
Then the closure $\overline{Y}$ of $Y$ in $h_P (X)$ is
 equivalent to the subpower Higson compactification
 $h_P (Y)$ of $Y$.
\end{prop}

\begin{proof}
By \cite[Theorem 3.5.5]{Eng},
 it suffices to show $\mbox{cl}_{h_P (X)} A \cap \mbox{cl}_{h_P (X)} B =\emptyset$
 if and only if $\mbox{cl}_{h_P (Y)} A \cap \mbox{cl}_{h_P (Y)} B =\emptyset$
 for every pair $A$, $B$ of closed subsets of $Y$.
It follows from Proposition \ref{P2-3} that
 the condition
 $\mbox{cl}_{h_P (X)} A \cap \mbox{cl}_{h_P (X)} B =\emptyset$
 if and only if the system $A,B$ diverges as a power function in $X$ and $A\cap B=\emptyset$.
This is equivalent to the condition that
 the system $A,B$ diverges as a power function in $Y$ and $A\cap B=\emptyset$
 since the metric on $Y$ is inherited from $X$.
\end{proof}

Let $(X, d)$ be a metric space.
For any $R>0$, a subset $Y$ of $X$ is called {\it $R$-dense} in $X$
 if $B_d (x ,R)\cap Y\neq \emptyset$ for every $x\in X$.

\begin{cor}[cf. \cite{Keesling}]
\label{C2-5}
Let $(X, d)$ be a proper metric space
 and let $Y$ be a subset of $X$.
If $Y$ is $R$-dense in $X$ for some $R>0$
 then $\overline{Y}\setminus Y =\nu_P X$. 
In particular, $\nu_P Y$ is homeomorphic to $\nu_P X$. 
\end{cor}

\begin{proof}
Let $x\in \nu_P X$.
Let $U$ and $V$ be open neighborhoods of $x$ in $h_P (X)$
 such that $x\in V\subset \overline{V}\subset U$.
Put $E=X\cap (h_P (X)\setminus U)$ and $F=X\cap \overline{V}$.
By Proposition \ref{P2-3},
 the system $E, F$ diverges as a power function
 since $\nu_P X \cap \overline{E}\cap \overline{F}=\emptyset$.
Thus, there are $\alpha, r >0$ such that
 $d(y, E)>|y|^{\alpha}$ whenever $y\in F$ and $|y|>r$.
Since $V\cap \nu_P X\neq \emptyset$ and $X$ is a proper metric space,
 we can take $z\in F$ so that $|z|>r$.
We may assume that $R<r^{\alpha}$.
Then $B_d(z, R )\cap E=\emptyset $, i.e.,
 $B_d(z, R)\subset U$.
Since $Y$ is $R$-dense in $X$,
 $\emptyset \neq Y\cap B_d(z, R) \subset Y\cap U$.
This fact leads that $x\in \overline{Y}$ since we can take $U$ as an arbitrarily small neighborhood of $x$.
Hence $\nu_P X$ is contained in $\overline{Y}\setminus Y$,
 i.e., $\overline{Y}\setminus Y = \nu_P X$.
By Proposition \ref{P2-4}, $\overline{Y}\setminus Y$ is homeomorphic to $\nu_P Y$.
Thus $\nu_P Y$ is homeomorphic to $\nu_P X$.
\end{proof}

%%%%%%%%%%%%%%%%%%%%%%%%%%%%%%%%%%%%%%%%%%%%%%%%%%%%%%%%%%%%%%%%%%%%%%%%%%%%

\section{Continua as subpower Higson coronas of $[0, \infty)$}
Throughout this section,
 $0$ is assumed to be the base point of the half-open interval $[0,\infty)$,
 that is, $x_0 =0\in [0, \infty)$.
A metric $d$ on $[0,\infty)$ defined by $d(x,y)=\sqrt{(x-y)^2}$
 is called {\it the usual metric} on $[ 0,\infty)$.
For notational simplicity,
 we identify $t$ with $|t|=d(x_0, t)$ for every $t\in [0, \infty)$
 when we consider the usual metric on $[0,\infty)$.
A {\it continuum} is a nonempty, compact and connected Hausdorff space.
A {\it subcontinuum} is a continuum which is a subset of a continuum.
A continuum is called {\it decomposable}
 if it can be represented as the union of two of its proper subcontinua.
A continuum which is not decomposable is said to be {\it indecomposable}.

Let $X=[0,\infty)$ be the half open interval 
 with the usual metric.
The indecomposability of the Stone-\v{C}ech remainder $\beta X\setminus X$
 and the Higson corona $\nu_H X$
 was proved in \cite{Bellamy} and in \cite{IT} respectively.
We show here that
 the subpower Higson corona $\nu_P X$ is also an indecomposable continuum.
Let $K$ be a proper closed subset of $\beta X \setminus X$
 with non-empty interior in $\beta X\setminus X$.
To see the indecomposability of $\beta X \setminus X$,
 Bellamy \cite{Bellamy} constructed a continuous map
 $f:X\to [0,1]$
 such that $f^{\ast}(K)=\{ 0, 1\}$,
 where $f^{\ast}:\beta X\to [0,1]$ is the extension of $f$.
Hence $K$ is disconnected.
We note that such the extension $f^{\ast}$
 to the Stone-\v{C}ech compactification always exists
 since $f$ is a bounded continuous map.
And then he constructed a continuous surjection
 from a given non-degenerate subcontinuum of $\beta X\setminus X$
 onto a given metric continuum.
Our strategy is essentially the same as the Bellamy's.
To see the indecomposability of $\nu_P X$,
 we will construct a map $\psi : X\to [0,1]$
 such that $\psi^{\ast} (K')=\{ 0, 1\}$ for a given proper closed subset $K'$ of $\nu_P X$
 with non-empty interior in $\nu_P X$,
 where $\psi^{\ast}:h_P X\to [0,1]$ is the extension of $\psi$.
In order to ensure that a map  $\psi:X\to [0,1]$
 has the extension $\psi^{\ast}:h_P X\to [0,1]$,
 we will carefully construct $\psi$ to be Higson subpower (cf. Proposition \ref{P2-1}).
Then we will construct a continuous surjection from $\nu_P X$
 onto the Higson type compactification $h_A (X)$
 for each $A\in \{ H, P, L\}$.

The following lemma plays an essential role to analyze the subpower Higson corona
 of the half-open interval $[0, \infty)$.

\begin{lem}\label{L3-1}
Let $X=[0,\infty)$ be the half open interval with the usual metric $d$.
Let $U$ and $V$ be non-empty disjoint open subsets of $h_P (X)$ such that
 $U\cap \nu_P X\neq \emptyset$ and $V\cap \nu_P X\neq \emptyset$.
Then there exist a natural number $k\geq 3$ and sequences $(a_n)_{n=1}^{\infty}$ and $(b_n)_{n=1}^{\infty}$
 with
$
0=b_0<a_1 <b_1 < \cdots <b_{n-1}<a_n <b_n <\cdots
$
 satisfying the following conditions:
\begin{enumerate}
\item
$(a_n)^{1/k}>2^{n}$,\label{s-1}
\item
$b_{n-1}+(b_{n-1})^{1/k}<a_n$,\label{s-2}
\item
$[b_{n-1}, a_n]\cap U \neq \emptyset$,\label{s-3}
\item
$b_n =a_n +(a_n)^{2/k}$,\label{s-4}
\item
$[a_n, b_n]\subset V\cap X$,\label{s-5}
\item
$a_n -b_{n-1}>2^{n-1}(a_n)^{1/k}$,\label{s-6}
\item
$b_n -a_n > 2^{n}(b_n)^{1/k}$.\label{s-7}
\end{enumerate}
\end{lem}

\begin{proof}
Let $z$ be a point in $V\cap \nu_P X$
 and let $W$ be an open neighborhood of $z$ in $h_P X$ such that
 $z\in W \subset \overline{W}\subset V$.
Then 
\[
\nu_P X \cap \overline{(X \setminus (V\cap X))}\cap \overline{(\overline{W}\cap X)}
\subset \nu_P X \cap (h_P (X)\setminus V) \cap \overline{W}
=\emptyset.
\]
By Proposition \ref{P2-3},
 the system $X \setminus (V\cap X)$, $\overline{W}\cap X$
 diverges as a power function.
Hence,
 there exist $\alpha>0$ and $r_0>0$ such that
\[
 d(t,X \setminus (V\cap X))\geq t^\alpha
\text{ whenever } t\in \overline{W}\cap X \text{ and } t \geq r_0.
\]
Take a natural number $k\geq 3$ so that $\alpha >2/k$.
Then the condition above implies that
\[
\tag*{(A)}\label{A}
 B_d (t, t^{2/k})\subset V\cap X
\text{ whenever } t\in \overline{W}\cap X \text{ and } t\geq r_0.
\]
Then we consider the following two inequalities:
\begin{align}\tag*{(B)}\label{B}
\begin{cases}
x^k >x,\\
x^2 >2^{k}(x+x^{2/k}).
\end{cases}
\end{align}
Since $k\geq 3$, there is $t_1 >2$ such that
 $x$ satisfies \ref{B} whenever $x>t_1$.
Then we take $a_1 \in W\cap X$ so that $a_1>\max\{ r_0, (t_1)^k \}$.
This is possible
 since the set $W\cap X$ is cofinal in $X$
 by the condition $W\cap \nu_P X\neq \emptyset$.
Hence, $a_1$ satisfies 
 the conditions $(\ref{s-1})$ and $(\ref{s-2})$.
Also, the condition $U\cap \nu_P X\neq \emptyset$ allows us
 to make $a_1$ satisfy the condition $(\ref{s-3})$.
Then we define $b_1$ by $(\ref{s-4})$.
Since $a_1 \in \overline{W}\cap X$ and $d(a_1, b_1)=(a_1)^{2/k}$,
 we have $b_1 \in B_d (a_1, (a_1)^{2/k})\subset V\cap X$ by \ref{A}.
Then the condition $(\ref{s-5})$ follows from the fact that
 the metric $d$ is geodesic.
To see that $a_1$ and $b_1$ satisfy the conditions $(\ref{s-6})$ and $(\ref{s-7})$,
 put $a_1 =x$.
Then $b_1 =x +x^{2/k}$ by $(\ref{s-4})$.
Since $a_1 >t_1$, we have $(a_1 -b_0)^k =x^k >x =((a_1)^{1/k})^k$
 and $(b_1 -a_1)^k =x^2 >2^{k}(x+x^{2/k}) =(2 (b_1)^{1/k})^k$ by \ref{B}.
Thus the conditions $(6)$ and $(7)$ are satisfied.

Suppose that $b_{i-1}<a_i <b_i$ have been constructed for $i< n$.
Then we consider the following two inequalities:
\begin{align*}\tag*{(C)}\label{C}
\begin{cases}
(x-b_{n-1})^k >2^{(n-1)k}x,\\
x^2 >2^{nk}(x+x^{2/k}).
\end{cases}
\end{align*}
Since $k\geq 3$, there exists $t_n >2^n$ such that
 $x$ satisfies \ref{C} whenever $x>t_n$.
As in the first step,
 we can take $a_n \in W\cap X$ with
 $a_n >\max\{ r_0, (t_n)^k \}$
 so that
 the conditions $(\ref{s-1})$--$(\ref{s-3})$ are satisfied.
Then we define $b_n$ by (\ref{s-4}).
The condition (\ref{s-5}) follows from \ref{A}
 and the fact that the metric $d$ is geodesic.
Finally, to see the conditions (\ref{s-6}) and (\ref{s-7}),
 we put $a_n =x$.
Then $b_n =x+x^{2/k}$ by (\ref{s-4}).
Since $a_n >t_n$,
 it follows from \ref{C} that
 $(a_n -b_{n-1})^k =(x-b_{n-1})^k >2^{(n-1)k}x =(2^{n-1}(a_n )^{1/k})^k$
 and
 $(b_n -a_n)^k =x^2 >  2^{nk}(x+x^{2/k}) =( 2^{n}(b_n)^{1/k})^k$.
The conditions (\ref{s-6}) and (\ref{s-7}) are satisfied.
\end{proof}

\begin{prop}\label{P3-2}
Let $X=[0,\infty)$ be the half open interval with the usual metric $d$.
If $K$ is a proper closed subset of $\nu_P X$ with non-empty interior in $\nu_P X$
 then $K$ is disconnected.
\end{prop}

\begin{proof}
Let $x \in \mbox{Int}_{\nu_P X}K$ and $y \in \nu_P X\setminus K$.
Let $U$ be an open neighborhood of $x$ in $h_{P} X$
 such that $\overline{U}\cap \nu_P X \subset K$.
Since $y\notin \overline{U}$,
 we can take an open neighborhood $V$ of $y$ in $h_P X$ such that
 $y\in V \subset \overline{V}\subset h_P X \setminus (K\cup \overline{U})$.
By Lemma \ref{L3-1}, 
 there exist a natural number
 $k\geq 3$ and sequences $(a_n)_{n=1}^{\infty}$ and $(b_n)_{n=1}^{\infty}$
 with
$
0=b_0<a_1 <b_1 < \cdots <b_{n-1}<a_n <b_n <\cdots
$
 satisfying the following conditions:
\begin{enumerate}
\item
$(a_n)^{1/k}>2^{n}$,\label{p-1}
\item
$b_{n-1}+(b_{n-1})^{1/k}<a_n$,\label{p-2}
\item
$[b_{n-1}, a_n]\cap U \neq \emptyset$,\label{p-3}
\item
$b_n =a_n +(a_n)^{2/k}$,\label{p-4}
\item
$[a_n, b_n]\subset V\cap X$,\label{p-5}
\item
$a_n -b_{n-1}>2^{n-1}(a_n)^{1/k}$,\label{p-6}
\item
$b_n -a_n > 2^{n}(b_n)^{1/k}$.\label{p-7}
\end{enumerate}

\setcounter{cn}{1}
\begin{clm}\label{clm3-2-1}
Let $u \in B_d(t, t^{1/k})$.
Then we have the following:
\begin{enumerate}
\item[(i)]
If $t\in [a_n, b_n)$ then
 $u\in (b_{n-1},  a_{n+1})$.
\item[(ii)]
If $t\in [b_n, a_{n+1})$ then
 $u\in (a_n, b_{n+1})$.
\end{enumerate}
\end{clm}

\noindent
{\itshape Proof of Claim \ref{clm3-2-1}.}\ 
Suppose that $t \in [a_n, b_n)$.
If $u<t$ then
 $u\geq t-t^{1/k}
 \geq a_n -(a_n )^{1/k}
 \geq a_n -2^{n-1}(a_n)^{1/k}
 >b_{n-1}
 $
 by the monotone increasing-property of the correspondence
 $s\mapsto s-s^{1/k}$ on $(1, \infty)$
 and $(\ref{p-6})$.
If $t<u$ then
 $u \leq t+t^{1/k}
 < b_n +(b_n)^{1/k}
 < a_{n+1}$
 by $(\ref{p-2})$.
Thus the condition (i) is satisfied.

Next suppose that $t\in [b_n, a_{n+1})$.
If $u<t$ then, as above, using $(\ref{p-7})$,
 $u\geq t-t^{1/k}
 \geq b_{n}-(b_{n})^{1/k}
 >b_{n}-2^{n}(b_{n})^{1/k}
 >a_{n}$.
If $u>t$ then
 $u\leq t+t^{1/k}
 < a_{n+1}+(a_{n+1})^{1/k}
 <a_{n+1}+(a_{n+1})^{2/k}
 =b_{n+1}$ by $(\ref{p-4})$.
Thus the condition (ii) is satisfied.
\hfill$\lozenge$\par\medskip

We define $\psi: X\to [0,1]$ by
\begin{align*}
\psi (t)=
\begin{cases}
0, & \mbox{if $b_{2n-2}\leq t <a_{2n-1}$},\\
\dfrac{t-a_{2n-1}}{b_{2n-1}-a_{2n-1}}, & \mbox{if $a_{2n-1} \leq t <b_{2n-1}$},\\
1, & \mbox{if $b_{2n-1}\leq t <a_{2n}$,}\\
\dfrac{b_{2n}-t}{b_{2n}-a_{2n}}, & \mbox{if $a_{2n} \leq t <b_{2n}$}.
\end{cases}
\end{align*}
Then it is easy to see that the function $\psi $ is well-defined and continuous.
Now we shall evaluate the diameter of $\psi (B_d(t, t^{1/k}))$.

\setcounter{cn}{2}
\begin{clm}\label{clm3-2-2}
If $a_{n} \leq t <a_{n+1}$ then $|\psi (t)-\psi (u)|<1/2^{n}$ for every $u \in B_d(t, t^{1/k})$.
\end{clm}

\noindent
{\itshape Proof of Claim \ref{clm3-2-2}.}\ 
Let $u \in B_d(t, t^{1/k})$.
By the definition of $\psi$,
 we check the claim by dividing it into four cases,
 $t\in [a_{2n-1}, b_{2n-1})$,
 $t\in [b_{2n-1}, a_{2n})$,
 $t\in [a_{2n}, b_{2n})$ and
 $t\in [b_{2n}, a_{2n+1})$.
We show here the first two cases.
The other two cases can be shown in a similar fashion.

Suppose that $t\in [a_{2n-1}, b_{2n-1})$.
By Claim \ref{clm3-2-1},
 it suffices to consider the three cases that $u\in (b_{2n-2}, a_{2n-1})$,
 $u\in [a_{2n-1}, b_{2n-1})$
 and $u\in [b_{2n-1}, a_{2n})$.
If $u \in (b_{2n-2}, a_{2n-1})$ then, using $(\ref{p-7})$,
\begin{align*}
|\psi (t)-\psi (u)| &=\psi (t)=\dfrac{t-a_{2n-1}}{b_{2n-1} -a_{2n-1}}
\leq \dfrac{t-u}{b_{2n-1} -a_{2n-1}}\\
 & \leq \dfrac{t^{1/k}}{b_{2n-1} -a_{2n-1}}
 < \dfrac{(b_{2n-1})^{1/k}}{b_{2n-1} -a_{2n-1}}\\
 & <\dfrac{(b_{2n-1})^{1/k}}{2^{2n-1}(b_{2n-1})^{1/k}}
 <\dfrac{1}{2^{2n-1}}.
\end{align*}
If $u \in [a_{2n-1}, b_{2n-1})$ then we have
\begin{align*}
|\psi (t)-\psi (u)| &=\dfrac{|t-u|}{b_{2n-1} -a_{2n-1}}
 \leq \dfrac{t^{1/k}}{b_{2n-1} -a_{2n-1}} 
  <\dfrac{1}{2^{2n-1}}.
\end{align*}
If $u\in [b_{2n-1}, a_{2n})$ then we have
\begin{align*}
|\psi (t)-\psi (u)| &=1-\psi (t)
 =1-\dfrac{t-a_{2n-1}}{b_{2n-1}-a_{2n-1}}
 =\dfrac{b_{2n-1}-t}{b_{2n-1} -a_{2n-1}}\\
 &\leq \dfrac{u-t}{b_{2n-1} -a_{2n-1}}
 \leq \dfrac{t^{1/k}}{b_{2n-1} -a_{2n-1}}
 <\dfrac{1}{2^{2n-1}}.
\end{align*}
Thus $|\psi (t)-\psi (u)|<1/2^{2n-1}$ whenever $t\in [a_{2n-1}, b_{2n-1})$.

Next suppose $t\in [b_{2n-1}, a_{2n})$. 
By Claim \ref{clm3-2-1},
 it suffices to consider the three cases that
 $u\in (a_{2n-1}, b_{2n-1})$,
 $u\in [b_{2n-1}, a_{2n})$
 and $u\in [a_{2n}, b_{2n})$.
Let $u\in (a_{2n-1}, b_{2n-1})$.
Since $u\geq t-t^{1/k} \geq b_{2n-1}-(b_{2n-1})^{1/k} $,
 it follows that
\[
\tag*{$(\dag)$}\label{dag} 
b_{2n-1}-u\leq (b_{2n-1})^{1/k}.
\]
So, using $(\ref{p-7})$ and $\ref{dag}$, we have
\begin{align*}
|\psi (t)-\psi (u)|&=1-\psi (u)
=1-\dfrac{u-a_{2n-1}}{b_{2n-1}-a_{2n-1}}
=\dfrac{b_{2n-1}-u}{b_{2n-1}-a_{2n-1}}\\
&
\leq \dfrac{(b_{2n-1})^{1/k}}{b_{2n-1} -a_{2n-1}}
<\dfrac{(b_{2n-1})^{1/k}}{2^{2n-1}(b_{2n-1})^{1/k}}
<\dfrac{1}{2^{2n-1}}.
\end{align*}
If $u\in [b_{2n-1}, a_{2n})$ then $|\psi (t)-\psi (u)|=1-1=0$.
Finally, let $u\in [a_{2n}, b_{2n})$.
Then, using $(\ref{p-4})$ and $(\ref{p-1})$, we have
\begin{align*}
|\psi (t)-\psi (u)|&=1-\psi(u)
 =1-\dfrac{b_{2n}-u}{b_{2n}-a_{2n}}
 =\dfrac{u-a_{2n}}{b_{2n}-a_{2n}}\\
& < \dfrac{u-t}{b_{2n}-a_{2n}}
 \leq \dfrac{t^{1/k}}{b_{2n}-a_{2n}}
 = \dfrac{t^{1/k}}{(a_{2n})^{2/k}}\\
& < \dfrac{(a_{2n})^{1/k}}{(a_{2n})^{2/k}}
=\dfrac{1}{(a_{2n})^{1/k}}
<\dfrac{1}{2^{2n}}
<\dfrac{1}{2^{2n-1}}.
\end{align*}
Thus $|\psi (t)-\psi (u)|<1/2^{2n-1}$ whenever $t\in [b_{2n-1}, a_{2n})$.
\hfill$\lozenge$\par\medskip

\setcounter{cn}{3}
\begin{clm}\label{clm3-2-3}
The map $\psi : X\to [0,1]$ is Higson subpower.
\end{clm}

\noindent
{\itshape Proof of Claim \ref{clm3-2-3}.}\ 
Let $\varepsilon >0$ and let $s:X\to \mathbb R_+$ be an asymptotically subpower function.
Then there exists $r >0$ such that $s(t)<t^{1/k}$ for every $t>r$.
Let $m$ be a natural number such that $1/2^{m}<\varepsilon/2$.
Put $t_0 =\max \{ r, a_m \}$.
Then,
 for every $t>t_0 $ and $u\in B_d(t, t^{1/k})$,
 we have $B_d(t, s(t))\subset B_d(t, t^{1/k})$
 and $|\psi (t)-\psi (u)|<1/2^m$
 by Claim \ref{clm3-2-2}.
Thus $\diam \psi (B_d(t, s(t)))
\leq \diam \psi (B_d(t, t^{1/k}))\leq 2/2^m <\varepsilon$
 for every $t>t_0$.
Hence $\psi$ is Higson subpower.
\hfill$\lozenge$\par\medskip

By Claim \ref{clm3-2-3} and Proposition \ref{P2-1},
 there exists an extension $\Psi: h_P X\to [0,1]$ of $\psi $.
By $(\ref{p-3})$,
 we obtain a sequence $(c_n)_{n=1}^{\infty}$ such that
 $b_n <c_n <a_{n+1}$ and $c_n \in U$.
Then $\psi (c_{2n})=0$ and $\psi (c_{2n-1})=1$ for every $n$.
Let $z_0 $ and $z_1$
 be cluster points of the sequences $(c_{2n})_{n=1}^{\infty}$ and $(c_{2n-1})_{n=1}^{\infty}$
 respectively.
Then both $z_0$ and $z_1$ are contained in $\overline{U}\cap \nu_P X \subset K$.
It follows from the continuity of $\Psi$ that $\Psi(z_0)=0$ and $\Psi(z_1)=1$.
Hence $z_0 \in \Psi^{-1}(0)\cap K\neq \emptyset$
 and $z_1 \in \Psi^{-1}(1)\cap K\neq \emptyset$.
However, $\Psi^{-1}((0,1))\cap K=\emptyset$.
Indeed, if there exists $z \in K$ such that $\Psi(z )=t\in (0,1)$
 then $z$ has a neighborhood $W \subset h_P X$ such that $\Psi(W)\subset (0,1)$.
Since $K\cap \overline{V}=\emptyset$,
 we may assume $W \cap V=\emptyset$.
Then
 we have $\Psi(W \cap X)=\psi (W \cap X)\subset (0,1)$,
 i.e.,
 $W \cap X \subset \psi^{-1}((0,1))$.
On the other hand, $\psi ^{-1}((0,1))\subset V\cap X$ by $(\ref{p-5})$ and our construction of $\psi $.
So we have $\emptyset \neq W \cap X\subset V\cap X$
 which implies that $W\cap V\neq \emptyset$, a contradiction.
Thus $\Psi(K)=\{ 0, 1\}$, i.e., $K$ is disconnected.
\end{proof}

As we have seen, the function $\psi : X\to [0,1]$ constructed above is Higson subpower.
However it is not Higson sublinear.
Indeed, take sequences $(a_n)_{n=1}^{\infty}$, $(b_n)_{n=1}^{\infty}$ and $k\geq 3$ as in Proposition \ref{P3-2}.
Let $\xi : \mathbb R_+ \to \mathbb R_{+}$ be the function defined by $\xi (t)=t^{2/k}$.
Since $k\geq 3$,
 it follows that $\xi$ is an asymptotically sublinear function.
Since $b_{n}=a_n +(a_{n})^{2/k}$, 
 the closed ball $ B_d(a_{n}, \xi(a_{n}))$ contains at least two points
 $a_{n}$ and $b_{n}$.
By our construction of $\psi$,
 we have $(\psi(a_n), \psi(b_n))\in \{ (0,1), (1,0)\}$,
 i.e., $\diam  \psi (B_d(a_{n}, \xi(a_{n}))) =1$ for every $n$.
Hence, $\psi $ is not Higson sublinear.

\begin{thm}\label{T3-3}
Let $X=[0,\infty)$ be the half open interval with the usual metric.
Then the subpower Higson corona $\nu_P X$ is a non-metrizable indecomposable continuum.
\end{thm}

\begin{proof}
Let $A_n =[n, \infty)\subset X$.
Then $\overline{A_n}$ is a continuum and $\overline{A_{n}}\supset \overline{A_{n+1}}$ for every $n\in \Bbb N$.
Note that $\nu_P X=\cap_{n=1}^{\infty} \overline{A_n}$.
Thus $\nu_P X$ is a continuum as the intersection of the decreasing sequence
 $(\overline{A_{n}})_{n=1}^{\infty}$ of continua.
Non-metrizability of $\nu_P X$ follows from the fact that
 $\nu_P X$ contains a copy of $\beta \mathbb N \setminus \mathbb N$
 \cite{KZ1}
 which has the cardinality $2^{ \mathfrak{c}}$  \cite{GJ}.
Finally, if $\nu_P X$ is expressed as the union of two non-degenerate proper closed subsets $K$ and $L$ of $\nu_P X$
 then both of $K$ and $L$ must have non-empty interiors in $\nu_P X$.
Thus both $K$ and $L$ are disconnected by Proposition \ref{P3-2}.
\end{proof}

The following is an example,
 given in \cite{IT},
 of a proper metric on $[0, \infty)$ which derives a decomposable Higson corona.
We show that the same example also derives a decomposable subpower Higson corona.

\begin{ex}\label{E3-4}
There exists a proper metric $\rho$ on the half open interval $[0,\infty)$
 such that the Higson subpower corona is a decomposable continuum.
Indeed, let $f: [0,\infty)\to \mathbb R^2$ be the embedding given by
 $f(t)=(t, t\sin t)$ and let $X=f([0,\infty))$.
Let $\rho$ be the metric on $X$ inherited from $\mathbb R^2$.
We shall show that the Higson subpower corona $\nu_P X$ of $(X, \rho)$
 is a decomposable continuum.
Put
\begin{align*}
Y&=\{ (x,y)\in X \mid y\geq 0\},\\
Z&=\{ (x,y)\in X \mid y\leq 0\},\\
A&=\{ (x,y)\in \mathbb R^2 \mid x\geq 0,\ |y|\leq x\},\\
B&=\{ (x,y)\in A \mid y\geq 0\},\\
C&=\{ (x,y)\in A \mid y\leq 0\}.
\end{align*}
Then $X=Y\cup Z$, $A=B\cup C$, $Y\subset B$ and $Z\subset C$.
For each $A\subset X$,
 $\overline{A}$ denotes the closure of $A$ in the subpower Higson compactification $h_P (X)$ of $X$.

First we note that
\[
\nu_P X =\overline{X}\setminus  X
 =(\overline{Y}\setminus Y)\cup (\overline{Z}\setminus Z).
\]
Indeed, $\overline{X}\setminus  X\supset
 (\overline{Y}\cup \overline{Z}) \setminus (Y\cup Z)
 =(\overline{Y}\setminus Y)\cup (\overline{Z}\setminus Z)$.
If $z\in \overline{X}\setminus  X$ has an open neighborhood $U$
 with $U\cap Y=\emptyset$
 then every neighborhood of $z$ must intersect with $Z$ by the density of $X$ in $\overline{X}$,
 i.e., $z\in \overline{Z}\setminus Z$.

Next we shall show that both $\overline{Y}\setminus Y$
 and $\overline{Z}\setminus Z$ are subcontinua of $\nu_P X$.
By Proposition \ref{P2-4}, $\overline{Y}\setminus Y$ is homeomorphic to $\nu_P Y$
(notation: $\overline{Y}\setminus Y \approx \nu_P Y$)
 and
 $\overline{Z}\setminus Z\approx \nu_P Z$.
It is easy to see that $Y$ is $R$-dense in $B$ for some $R>0$.
Hence $\nu_P Y$ is homeomorphic to $\nu_P B$ by Corollary \ref{C2-5}.
Thus, we have $\overline{Y}\setminus Y \approx \nu_P B$.
Similarly, we have $\overline{Z}\setminus Z \approx \nu_P C$.
For each $n\in \Bbb N$,
 let $B_n=\{ (x,y)\in B \mid  x\geq n\}$.
Then $\nu_P B =\cap_{i=1}^{\infty} \mbox{cl}_{\nu_P B}(B_n)$.
Hence $\nu_P B$ is a continuum 
 as the intersection of the decreasing sequence $(\mbox{cl}_{\nu_P B}(B_n))_{n=1}^{\infty}$
 of continua.
Similarly, it follows that $\nu_P C$ is a continuum.
Thus, both $\overline{Y}\setminus Y$
 and $\overline{Z}\setminus Z$ are subcontinua of $\nu_P X$.

Finally,
 we shall show that both $\overline{Y}\setminus Y$ and $\overline{Z}\setminus Z$
 are proper subcontinua of $\nu_P X$.
Let $D=\{ f(t)\mid t=\pi /2 +2n\pi , n\in \mathbb N \}$ and
 $E=\{ f(t)\mid t=3\pi /2 +2n\pi , n\in \mathbb N \}$.
Then $ \overline{Y}\setminus Y \supset \overline{D} \setminus D\neq \emptyset$.
It is easy to see that 
 the systems $D, Z$ diverges as a power function.
Hence we have $\overline{D} \cap \overline{Z}\cap \nu_P X=\emptyset$ by Proposition \ref{P2-3},
 i.e.,
 $ (\overline{Y}\setminus Y)\setminus  (\overline{Z}\setminus Z) \supset \overline{D} \setminus D \neq \emptyset$.
Similarly, we have
 $ (\overline{Z}\setminus Z)\setminus  (\overline{Y}\setminus Y) \supset \overline{E} \setminus E \neq \emptyset$.
Thus, both
 $\overline{Y}\setminus Y$ and $\overline{Z}\setminus Z$ are proper subcontinua of $\nu_P X$
 with $\nu_P X =(\overline{Y}\setminus Y)\cup (\overline{Z}\setminus Z)$.
Hence $\nu_P X$ is a decomposable continuum.
\end{ex}
\par\medskip

Let $K\subset X$ be a continuum
 and let $a, b\in K$ be distinct two points.
If there is no proper subcontinuum of $K$ containing both $a$ and $b$
 then $K$ is said to be {\it irreducible between $a$ and $b$}.
 
The following lemma was proved by Bellamy \cite[Lemma 1]{Bellamy}
 for the Stone-\v{C}ech compactification of $[0,\infty)$.
The proof is valid for any compactification of $[0,\infty)$.
Here we give a proof for the reader's sake.

\begin{lem}\label{L3-5}
Every compactification $\alpha ([0,\infty))$ of the half-open interval $[0,\infty)$
 is an irreducible continuum between $0$ and every point $z$ of $\alpha ([0,\infty))\setminus [0, \infty)$.
\end{lem}

\begin{proof}
The connectivity of $\alpha ([0,\infty))$ is obvious.
Let $z \in \alpha  ([0,\infty))\setminus [0,\infty)$.
Suppose that there exists a proper closed subset $K$ of $\alpha ([0,\infty))$
 such that $0, z \in K$.
Since $\alpha ([0,\infty)) \setminus K$ is a non-empty open subset of $\alpha ([0,\infty))$,
 there exists a point $t\in [0,\infty)\cap (\alpha ([0,\infty)) \setminus K)$.
Then $U=K\cap [0,t)=K\cap [0,t]$ is a non-empty closed and open subset of $K$
 since $0\in U$.
On the other hand, $z \in K\setminus U\neq \emptyset$,
 that is, $K$ is not connected.
\end{proof}

\begin{thm}\label{T3-6}
Let $X=[0,\infty)$ be the half open interval with the usual metric $d$.
For each $A\in \{ H, P, L\}$,
 there exists a continuous surjection $\xi_{A} : \nu_P X \to h_A (X)$
 from the subpower Higson corona $\nu_P X$ onto the Higson type compactification $h_A (X)$.
\end{thm}

\begin{proof}
We shall construct a continuous surjection $\xi_{H} : \nu_P X \to h_H (X)$
 from the subpower Higson corona $\nu_P X$ onto the Higson compactification $h_H (X)$.
Then, for each $A\in \{ P, L\}$,
 a required surjection is obtained by the composition
 $\alpha^{H}_{A}\circ \xi_{H}:\nu_P X \to h_A (X)$
 where $\alpha^{H}_{A}: h_{H}(X) \to h_{A}(X)$ is the canonical surjection
 assured by the relation $h_{H}(X)\succeq h_{A}(X)$.

Let $x$ and $y$ be distinct two points in $\nu_P X$.
Let $U$ and $V$ be open subsets of  $h_P (X)$ such that
 $x\in U$, $y\in V$ and $\overline{U}\cap \overline{V}=\emptyset$.
Then, by Lemma \ref{L3-1}, 
there exist a natural number $k\geq 3$ and sequences $(a_n)_{n=1}^{\infty}$ and $(b_n)_{n=1}^{\infty}$
 with
$
0=b_0<a_1 <b_1 < \cdots <b_{n-1}<a_n <b_n <\cdots
$
 satisfying the following conditions:
\begin{enumerate}
\item
$(a_n)^{1/k}>2^{n}$,\label{t-1}
\item
$b_{n-1}+(b_{n-1})^{1/k}<a_n$,\label{t-2}
\item
$[b_{n-1}, a_n]\cap U \neq \emptyset$,\label{t-3}
\item
$b_n =a_n +(a_n)^{2/k}$,\label{t-4}
\item
$[a_n, b_n]\subset V\cap X$,\label{t-5}
\item
$a_n -b_{n-1}>2^{n-1}(a_n)^{1/k}$,\label{t-6}
\item
$b_n -a_n > 2^{n}(b_n)^{1/k}$.\label{t-7}
\end{enumerate}
Then we define $\theta: X\to X$ by
\begin{align*}
\theta (t)=
\begin{cases}
0, & \mbox{if $b_{2n-2}\leq t <a_{2n-1}$},\\
\dfrac{(2n-1)(t-a_{2n-1})}{b_{2n-1}-a_{2n-1}}, & \mbox{if $a_{2n-1} \leq t <b_{2n-1}$},\\
2n-1, & \mbox{if $b_{2n-1}\leq t <a_{2n}$,}\\
\dfrac{(2n-1)(b_{2n}-t)}{b_{2n}-a_{2n}}, & \mbox{if $a_{2n} \leq t <b_{2n}$}.
\end{cases}
\end{align*}
Let $\psi:X\to [0, 1]$ be the map constructed in Proposition \ref{P3-2},
 regarded as a map of $X$ into $X$, i.e., $\psi:X\to X$.
Then
 it follows from our construction that $\theta(t)=(2n-1)\psi (t)$
 whenever $b_{2n-2} \leq t <b_{2n}$.

Let $\ell : \mathbb R_+ \to \mathbb R_+$ be a function defined by
$\ell(t)=t/2^t$.
On should note that $\ell$ is monotonically decreasing on $(1/\log 2, \infty)$.
In particular, $\lim_{t\to \infty}\ell(t)=0$.
Then we have the following.

\setcounter{cn}{1}
\begin{clm}\label{clm3-6-1}
If $n\geq 3$ and
 $t\in [a_{2n-1}, a_{2n+1})$
 then $|\theta (t)-\theta (u)|<\ell(2n-1)$ for every $u \in B_d(t, t^{1/k})$.
\end{clm}

\noindent
{\itshape Proof of Claim \ref{clm3-6-1}.}\ 
Let $u \in B_d(t, t^{1/k})$.
If $t\in [a_{2n-1}, a_{2n})$ then $u \in (b_{2n-2}, b_{2n})$
 by Claim \ref{clm3-2-1}.
Then $\theta (t)=(2n-1)\psi(t)$ and $\theta (u)=(2n-1)\psi (u)$.
By Claim \ref{clm3-2-2}, we have
$|\theta (t)-\theta (u)|=(2n-1)|\psi(t)-\psi(u)|
 <(2n-1)/2^{2n-1}=\ell(2n-1)$.

Suppose that $t\in [a_{2n}, b_{2n})$.
Then $\theta (t)=(2n-1)\psi(t)$.
By Claim \ref{clm3-2-1}, $u \in (b_{2n-1}, a_{2n+1})$.
If $u\in (b_{2n-1}, b_{2n})$ then $\theta (u)=(2n-1)\psi (u)$.
If $u\in [b_{2n}, a_{2n+1})$
 then $\theta (u)=(2n-1)\psi(u)$ since $\theta (u)=\psi(u)=0$.
Hence, we obtain $|\theta (t)-\theta (u)|<\ell(2n-1)$ as above.

Now suppose that $t\in [b_{2n}, a_{2n+1})$.
In this case, $\theta(t)=\psi(t)=0$.
So we have $\theta(t)=(2n-1)\psi(t)$.
By Claim \ref{clm3-2-1}, $u \in (a_{2n}, b_{2n+1})$.
If $u\in (a_{2n}, b_{2n})$
 then $\theta (u)=(2n-1)\psi (u)$.
If $u\in [b_{2n}, a_{2n+1})$
 then $\theta (u)=(2n-1)\psi (u)$
 since $\theta(u)=\psi(u)=0$.
Hence, we have $|\theta (t)-\theta (u)|<\ell(2n-1)$ as above.
Finally, if $u \in [a_{2n+1}, b_{2n+1})$
 then, by the conditions $(\ref{t-4})$ and $(\ref{t-1})$, we have
\begin{align*}
|\theta(t)-\theta (u)|
&=\theta (u)
 =\dfrac{(2n+1)(u-a_{2n+1})}{b_{2n+1}-a_{2n+1}}
 =\dfrac{(2n+1)(u-a_{2n+1})}{(a_{2n+1})^{2/k}}\\
&< \dfrac{(2n+1)t^{1/k}}{(a_{2n+1})^{2/k}}
 <\dfrac{(2n+1)(a_{2n+1})^{1/k}}{(a_{2n+1})^{2/k}}
 =\dfrac{2n+1}{(a_{2n+1})^{1/k}}\\
&<\dfrac{2n+1}{2^{2n+1}}
=\ell (2n+1).
\end{align*}
We note that $\ell(2n+1)<\ell(2n-1)$
 since $\ell$ is decreasing on $(1/\log 2, \infty)$
 and $2n-1\geq 5>1/\log 2\fallingdotseq 3.3$.
Thus we have $|\theta(t)-\theta (u)|<\ell (2n-1)$.
\hfill$\lozenge$\par\medskip

\setcounter{cn}{2}
\begin{clm}\label{clm3-6-2}
If $n\geq 3$ and
 $t\geq a_{2n-1}$
 then $|\theta (t)-\theta (u)|<\ell(2n-1)$ for every $u \in B_d(t, t^{1/k})$.
\end{clm}
\noindent
{\itshape Proof of Claim \ref{clm3-6-2}.}\
Since $\ell$ is monotonically decreasing on $(1/\log 2, \infty)$
 and $2n-1\geq 5>1/\log 2\fallingdotseq 3.3$,
 this is an easy consequence of Claim \ref{clm3-6-1}.
\hfill$\lozenge$\par\medskip

\setcounter{cn}{3}
\begin{clm}\label{clm3-6-3}
For every $f\in C_{H}(X)$, the composition $f\circ \theta : X\to \mathbb R$
 is Higson subpower.
\end{clm}

\noindent
{\itshape Proof of Claim \ref{clm3-6-3}.}\ 
Let $f: X\to \mathbb R$ be a continuous Higson map
 and let $\varepsilon >0$ be given.
If $t\geq a_5$ then 
$
|\theta (t)-\theta(u)|
<\ell(5)
$
 for every $u\in B_d(t, t^{1/k})$ by Claim \ref{clm3-6-2}.
Since $f$ is Higson,
 there exists $K >0$ such that 
\[\tag{a}\label{a}
\diam f(B_d(t, \ell(5) ))<\varepsilon
 \text{ for every } t>K.
\]
By the compactness of the interval $[0, K]$
 and the continuity of $f$,
 we can take $\delta >0$ so that
\[\tag{b}\label{b}
|f(t)-f(u)|<\varepsilon /3 \text{ whenever } t \in [0, K] \text{ and } |t-u|< 3\delta.
\]
Without loss of generality,
 we may assume that $\delta <\ell(5)$.
Since $\lim_{t\to \infty}\ell(t)=0$,
 we can take an integer $n_0 \geq 3$ so that
$\ell (2n_0-1) <\delta/2$.

Let $s: X \to \mathbb R_+$ be an asymptotically subpower function.
We take $t_1 >0$ so that
 $s(t)<t^{1/k}$ for every $t>t_1$.
Then we put $t_0=\max\{ t_1, a_{2n_0 -1}\}$.

Let $t>t_0$.
It follows from $t>t_1$ that
$
\theta (B_d(t, s(t)))
\subset 
\theta (B_d(t, t^{1/k})).
$
Since $t>a_{2n_0 -1}$,
 it follows from Claim \ref{clm3-6-2}
 that
$|\theta (t)-\theta (u)|<\ell (2n_0 -1)$ for every $u\in B_d (t , t^{1/k})$,
 i.e.,
$\theta (B_d(t, t^{1/k}))
\subset B_d (\theta(t), 2\ell (2n_0 -1)).
$
Since $\ell (2n_0-1) <\delta/2$, we have
$
B_d (\theta(t), 2\ell (2n_0 -1))
\subset  B_d (\theta(t), \delta).
$
As a summary, we have
\[
\theta (B_d(t, s(t)))
\subset B_d (\theta(t), \delta).
\]
Suppose that $\theta (t)\leq K$.
It follows from the condition (\ref{b}) that
\[
\diam f\circ \theta (B_d(t, s(t)))
\leq
\diam f( B_d (\theta(t), \delta))
<\varepsilon.
\]
Now suppose that $\theta (t)>K$.
Since $\delta <\ell (5)$,
 we have
$ B_d(\theta(t), \delta)
\subset B_d(\theta(t), \ell (5))
$.
By the condition (\ref{a}), we have
\[
\diam f\circ \theta(B_d(t, s(t)))
\leq
\diam f(B_d(\theta(t), \ell (5)))
<\varepsilon.
\]
Hence $\diam  f\circ \theta (B_d(t, s(t))) <\varepsilon$ for every $t>t_0$.
\hfill$\lozenge$\par\medskip

\setcounter{cn}{4}
\begin{clm}\label{clm3-6-4}
The map $\theta :X\to X$ has an extension $\Theta:h_P (X)\to h_H (X)$.
\end{clm}

\noindent
{\itshape Proof of Claim \ref{clm3-6-4}.}\ 
Let $e:X\to \prod_{f\in C_H (X)} I_f$
 be the evaluation map.
Since $h_H (X)$ is equivalent to $\overline{e(X)}$,
 it suffices to see that
 $e\circ \theta : X\to \prod_{f\in C_H (X)} I_f$ can be extended to
 $h_P (X) \to \overline{e(X)}$.
For every $f\in C_H (X)$,
 the composition $f\circ \theta : X\to I_f$ is Higson subpower by Claim \ref{clm3-6-3},
 so it can be extended to the map
 $\widehat{f\circ \theta}: h_P (X) \to I_f$
 by Proposition \ref{P2-1}.
Hence $e\circ \theta$ can be extended to the map
 $\widehat{e\circ \theta}:h_P (X) \to \prod_{f\in C_P (X)} I_f$
 so that $\widehat{e\circ \theta}(h_P (X))\subset \overline{e(X)}$.
\hfill$\lozenge$\par\medskip

Let $(c_n)_{n=1}^{\infty}$ and $(d_n)_{n=1}^{\infty}$ be sequences such that
 $c_n \in [b_{2n-2}, a_{2n-1}]$
 and
 $d_n \in [b_{2n-1}, a_{2n}]$
 and let $p$ and $q$ be cluster points of $(c_n)_{n=1}^{\infty}$ and $(d_n)_{n=1}^{\infty}$
 respectively.
Put $C= \{ c_n\mid n\in \mathbb N\}$
 and $D=\{ d_n\mid n\in \mathbb N\}$.
Then the system $C, D$ diverges as a power function by (\ref{t-1}) and (\ref{t-7}).
Hence $p\in \overline{C}\cap \nu_P X$ and $q\in \overline{D}\cap \nu_P X$
 are distinct two points in $\nu_P X$.
By Claim \ref{clm3-6-4},
 there exists an extension $\Theta:h_P (X)\to h_H (X)$ of $\theta$.
Then $\Theta (p)=0$ and $\Theta (q)\in \nu_H X$.
Thus $\Theta (\nu_P X)$ is a continuum
 containing $0$ and $\Theta (q)\in \nu_H X$.
So,  $\Theta (\nu_P X)$ must coincide with $h_H (X)$ by Lemma \ref{L3-5}.
Hence the map $\xi_{H}=\Theta|\nu_P X :\nu_P X \to h_H (X)$ is a required surjection.
\end{proof}

\begin{cor}\label{C3-7}
Let $X=[0,\infty)$ be the half open interval with the usual metric $d$.
For each $A\in \{ H, P, L\}$,
 there exists a continuous surjection $\eta_{A} : \nu_H X \to h_A (X)$
 from the Higson corona $\nu_H X$ onto the Higson type compactification $h_A (X)$.
\end{cor}

\begin{proof}
Let $A\in \{ H, P, L\}$
 and let $\xi_A : \nu_P X \to h_A (X)$
 be a continuous surjection given by Theorem \ref{T3-6}.
By the relation $h_{H}(X)\succeq h_{P}(X)$,
 there is the canonical surjection $\alpha: h_H (X)\to h_{P}(X)$,
 such that $\alpha|_X =\text{id}_X$. 
Since $X$ is dense in $h_{H}(X)$
 and $\alpha|_X =\text{id}_X$,
 it follows that $\alpha (\nu_H X)=\nu_P X$.
Thus the composition
 $\eta_A =\xi_A \circ \alpha |_{\nu_H X} : \nu_H X \to h_A (X)$
 is a required surjection.
\end{proof}

\par\noindent
{\bf Question 1.}
Does Proposition \ref{P3-2} hold for sublinear Higson corona?
In particular, is the sublinear Higson corona of the half open interval with the usual metric
 an indecomposable continuum?
\par\medskip

Let $W$ be a non-degenerate subcontinuum of $\beta [0, \infty)\setminus [0,\infty)$
 and let $M$ be a metric continuum.
It is known \cite{Bellamy} that
 there exist continuous surjections $f: W\to \beta [0,\infty)$ and $g: W\to M$.
\par\medskip

\par\noindent
{\bf Question 2.}
Let $X=[0,\infty)$ be the half open interval with the usual metric $d$.
Let $W$ be a non-degenerate subcontinuum of $\nu_P(X)$
 and let $M$ be a metric continuum.
Do there exist continuous surjections $f: W\to h_P(X)$ and $g: W\to M$?
\par
\medskip
\noindent
{\bf Acknowledgement}
\par
I would like to express my sincere gratitude to the referee
 for his/her comments and valuable suggestions which helped to improve the manuscript.

\par\bigskip

\noindent
Department of Engineering Science,\\
National Institute of Technology,\\
Niihama College, Niihama, 792-8580, Japan\\
\noindent
E-mail address: iwamoto@sci.niihama-nct.ac.jp
%
%%%%%%%%%%%%%%%%%%%%%%%%%%%%%%%%%%%%%%%%%%%%%%%%%%%%%%%%%%%%%%%%%%%%%%%%%%%%%%%%%%%%%%%%%%%%%%%%
\end{document}